\documentclass[11pt, a4paper]{article}

\usepackage{amsmath, amsthm, amsfonts, amssymb}

\usepackage{setspace}
\usepackage{fullpage}
\usepackage{enumitem}
\usepackage{graphicx}

\usepackage{floatpag}
\usepackage[dvipsnames]{xcolor}

\usepackage[initials]{amsrefs} 

\usepackage{float}
\usepackage{caption}
\usepackage{subcaption}

\usepackage{hyperref}

\hypersetup{colorlinks=true,
	citecolor=blue,
	filecolor=blue,
	linkcolor=blue,
	urlcolor=blue
}
\usepackage{cleveref}

\usepackage{comment}
\usepackage{tikz}
\newtheorem{theorem}{Theorem} 
\newtheorem{lemma}[theorem]{Lemma}

\newcommand{\Up}{{\rm Up}}
\newcommand{\cH}{\mathcal{H}}

\newcommand{\ex}{{\rm  ex}}

\begin{document}
	\title{A class of graphs of  zero Tur\'an density in a hypercube
	\date{\vspace{-5ex}}
	\author{
		Maria Axenovich
	\thanks{Karlsruhe Institute of Technology, Karlsruhe, Germany;
			email:
			\mbox{\texttt{maria.aksenovich@kit.edu}}
}
}}
			
	\maketitle

	\begin{abstract}
	A graph is cubical if it is a subgraph of a hypercube. For a cubical graph $H$ and a hypercube $Q_n$, $\ex(Q_n, H)$ is the largest number of edges in an $H$-free subgraph of $Q_n$.
If $\ex(Q_n, H)$ is at least  a positive proportion of the number of edges in $Q_n$, $H$ is said to have a positive Tur\'an density in a hypercube or simply a positive Tur\'an density; otherwise it has a zero Tur\'an density.   Determining $\ex(Q_n, H)$ and even identifying whether $H$ has a positive or a zero Tur\'an density remains a widely open question for general $H$. 
By relating  extremal numbers in a hypercube  and certain corresponding hypergraphs, Conlon found a large class of cubical graphs, ones having  so-called partite representation, that have a zero Tur\'an density. He raised a question whether this gives a characterisation, i.e., whether a cubical graph has zero Tur\'an density if and only if it has partite representation. Here, we show that, as suspected by Conlon, this is not the case. 
We give an example of a class of cubical graphs which have no partite representation, but on the other hand, have a zero Tur\'an density.
In addition, we show that any graph whose every block has partite representation has  a zero Tur\'an density in a hypercube.	
	\end{abstract}

	\section{Introduction}
	
	A {\it hypercube} $Q_n$ with a ground set $X$ of size $n$, is a graph on a vertex set $\{A: A\subseteq X\}$ and an edge set consisting of all pairs $\{A,B\}$, where $A\subseteq B$ and $|A|=|B|-1$.  Unless specified, $X=[n]$, where $[n]= \{1, \ldots, n\}$. We often identify vertices of $Q_n$ with binary vectors that are indicator vectors of respective sets.  If a graph is a subgraph of $Q_n$, for some $n$, it is called {\it cubical}. We denote the number of vertices and the number of edges in a graph $H$ by  $|H|$ and $||H||$, respectively. The degree of a vertex $y$ in a graph $H$ is denoted $d(y)$ or $d_H(y)$. A {\it block} in a graph is a maximal two-connected subgraph or a bridge.  We shall need the notion of layers.  The $i$th {\it vertex layer} of $Q_n$,  denoted $V_i$, is the set of vertices  $\binom{[n]}{i}$,  $i=0, \ldots, n$.  The $i$th {\it edge-layer}  $L_i$ of $Q_n$ is a graph  induced by $V_i\cup V_{i-1}$, $i=1, \ldots, n$.  \\
	
For a graph $H$, let the {\it extremal number}  of $H$ in $Q_n$, denoted 	$\ex(Q_n, H)$,  be the largest number of edges in a subgraph $G$ of  $Q_n$ such that there is no subgraph of $G$ isomorphic to $H$. 
A graph $H$ is said to have {\it zero Tur\'an density in a hypercube} if $\ex(Q_n, H) = o(||Q_n||)$. Otherwise, we say that $H$ has  a {\it positive Tur\'an density in a hypercube}. Note that by using a standard double counting argument,  the sequence $\ex(Q_n, H)/||Q_n||$ is non-increasing, thus the above density notions are well-defined. When clear from context, we simply  say Tur\'an density instead of Tur\'an density in a hypercube. 
The behaviour of the function $\ex(Q_n, H)$ is not well understood in general and it is not even known what graphs have positive or zero Tur\'an density.
Currently, the only known cubical graphs of positive Tur\'an density are those containing a $4$- or a $6$-cycle as a subgraph, \cite{chung, conder}, and one special graph of girth $8$ \cite{AMW}.
Conlon \cite{C} observed a connection between extremal numbers in a hypercube and classical extremal numbers for uniform hypergraphs. That permitted to determine  large specific  classes of  graphs, such as for example subdivisions,  with zero Tur\'an density,  see \cite{AMW}.  For more results on extremal numbers in a hypercube, see  \cite{baber, balogh, AKS,TW, O}.\\

	A graph  $H$  has  a {\it $k$-partite representation}  $\cH$  if  for some $n$, $H$ is isomorphic to a subgraph $H'$ of the $k$th layer $L_k$ of $Q_n$ such that   $V(H') \cap V_k$ is an edge-set of a $k$-partite $k$-uniform hypergraph $\cH$.  The graph $H'$ is called {\it representing}.  If $H$ has a $k$-partite representation for some $k$, we say that $H$ has a {\it partite representation}. Here, a $k$-uniform hypergraph is $k$-partite if the vertex set can be partitioned into $k$ parts such that each hyperedge has exactly one vertex in each part. Here, we omit the brackets and commas denoting sets, when clear from context, i.e.,  for a set $\{1,2\}$ we simply write $12$.   For example, if $H$ is an $8$-cycle, it has a $2$-partite representation  $\cH$ with hyperedges 
$12, 23, 34, 14$  corresponding to an $8$-cycle  with vertices $1, 12, 2, 23, 3, 34, 4, 14, 1$, in order. 
For a $k$-uniform hypergraph $\cH$, $\ex_k(n, \cH)$ denotes the largest number of edges in a $k$-uniform $n$-vertex hypergraph with no subgraph isomorphic to $\cH$.  \\

Using a theorem by Erd\H{o}s \cite{E64}, that states that $\ex_k(n, \cH) = o_n(n^k)$ for any $k$-partite hypergraph $\cH$, 
Conlon \cite{C} proved that if $H$ is a graph that has partite representation, then  $\ex(Q_n, H) = o(||Q_n||)$.
In the same paper Conlon \cite{C} raised a question whether  cubical graphs that  have  no partite representation have positive Tur\'an density.   Here, we show that it is not the case, i.e. having a partite representation is not a characterisation for zero Tur\'an density in a hypercube.  For that, we construct a family of cubical  graphs that have no partite representation, but have zero Tur\'an density.  These graphs are formed by two copies of so-called {\it theta}-graphs that we define below.\\

For a graph $G$, let $G(1)$ be a $1$-subdivision of $G$, i.e., a graph $G'$ with a vertex set $V(G) \cup E(G)$  and  an edge set $\{ue, ev:  e=uv\in E(G)\}$. 
We call the vertices from $V(G)$ in $G'$,  the {\it poles} of $G'$ and other vertices, the {\it subdivision vertices}.
Let $K_{s,t}$ denote a complete bipartite graph with parts of sizes $s$ and $t$, respectively.  We shall be considering the $1$-subdivision of $K_{q,2}$, $q\geq 3$. We shall call the two vertices of degree $q$ in $K_{q, 2}(1)$, the {\it main poles}.
Note that $K_{q,2}(1)$ is also referred to as a {\it theta graph} with  $q$ legs of length $4$. Here the legs are paths with endpoints being main poles. We shall use a shorter notation $\Theta(q)$ for $K_{q, 2}(1)$, when appropriate. If a graph $G$ contains a subgraph isomorphic to $H$, we call such a subgraph a {\it copy of} $H$.\\

Let $H(q)$ be a union of two copies of $\Theta(q)$ sharing exactly one vertex that is a main pole of one copy and a subdivision vertex in another copy, see Figure \ref{H-graph}.\\

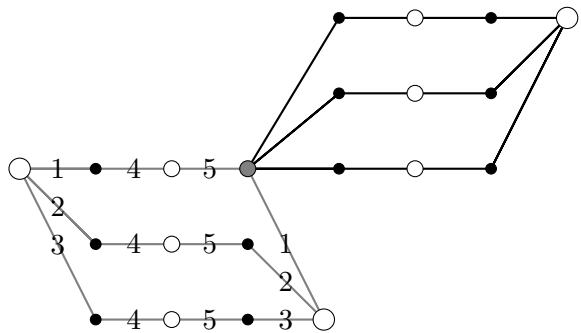
\begin{figure}[H]
\begin{center}
	\begin{tikzpicture}[scale=1.0]
	\newcommand\x{4.2}
	\newcommand\y{1}
	\newcommand\z{-1}
\draw[gray][thick] (0,1) --node[black] {2} (1,0) -- (0,1) -- node[black] {1} (1,1) -- (0,1) -- node[black] {3} (1,-1);
\draw [gray][thick](1,0) -- node[black] {4} (2,0)[gray] --node[black] {5 }(3,0)-- node[black] {2} (4,\z);
\draw [gray][thick](1,1)-- node[black] {4} (2,1)--node[black] {5} (3,1)-- node[black] {1} (4,\z);
\draw [gray][thick](1,-1)-- node[black] {4} (2,-1)-- node[black] {5} (3,-1)-- node[black] {3} (4,\z);

\draw[thick](3,1)--(\x,0+\y)--  (3,1)--(\x,1+\y)--(3,1)--(\x,2+\y);
\draw[thick](\x+3,2+\y)--(\x+2,0+\y)--(\x+3,2+\y)--(\x+2,1+\y)--(\x+3,2+\y)--(\x+2,2+\y);
\draw[thick](\x,1+\y)--(\x+1,1+\y)--(\x+2,1+\y);
\draw[thick](\x,2+\y)--(\x+1,2+\y)--(\x+2,2+\y);
\draw[thick](\x,0+\y)--(\x+1,0+\y)--(\x+2,0+\y);

\draw[fill=white] (0,1) circle (4pt);
\draw[fill=black] (1,0) circle (2pt);
\draw[fill=black] (1,1) circle (2pt);
\draw[fill=black] (1,-1) circle (2pt);

\draw[fill=white] (2,0) circle (3pt);
\draw[fill=white] (2,1) circle (3pt);
\draw[fill=white] (2,-1) circle (3pt);

\draw[fill=black] (3,0) circle (2pt);
\draw[fill=gray] (3,1) circle (3pt);
\draw[fill=black] (3,-1) circle (2pt);
\draw[fill=white] (4,\z) circle (4pt);

\draw[fill=black] (\x,0+\y) circle (2pt);
\draw[fill=black] (\x,1+\y) circle (2pt);
\draw[fill=black] (\x,2+\y) circle (2pt);

\draw[fill=white] (\x+1,0+\y) circle (3pt);
\draw[fill=white] (\x+1,1+\y) circle (3pt);
\draw[fill=white] (\x+1,2+\y) circle (3pt);

\draw[fill=black] (\x+2,0+\y) circle (2pt);
\draw[fill=black] (\x+2,1+\y) circle (2pt);
\draw[fill=black] (\x+2,2+\y) circle (2pt);

\draw[fill=white] (\x+3,2+\y) circle (4pt);

\end{tikzpicture}
	\end{center}
	\caption{A graph $H(3)$ with larger vertices being poles in  the respective copies of theta graphs as well as a nice coloring of one theta graph}
	\label{H-graph}
\end{figure}

Marquardt \cite{M} showed that $H(q)$ has no partite representation for $q=3$. Here, we give a slightly different proof of this fact for any $q\geq 3$  and show that $H(q)$ has a zero Tur\'an density. In doing so, we prove a result of an independent interest that gives  a new class of cubical graphs of zero Tur\'an density.

\begin{theorem}\label{block}
Let $H$ be a graph whose every block has a partite representation. Then $H$ has a zero Tur\'an density in a hypercube.
\end{theorem}

\begin{theorem}\label{main}
For any $q\geq 3$ the graph $H(q)$ is cubical, has no partite representation, but has a zero Tur\'an density in a hypercube.
\end{theorem}

\section{Proofs of main results}

We shall need the following lemmas.  The first one gives a property of graphs with partite representations, formulated by Marquardt  \cite{M}, we include it here for completeness. The second one is  basically a result of  Conlon \cite{C}, but with a slightly different averaging argument that  also gives  a more precise statement  about an embedding of a graph in a layer in two different ways. Finally, the last one is the main lemma needed for Theorem \ref{block}. If a graph $H$ has a $k$-partite representation with representing graph $H'$,  then the vertices of $H$ corresponding to $V(H')\cap V_k$ are called {\it top vertices} with respect to this representation, all other vertices are called {\it bottom vertices}.  Note that if a graph has a $k$-partite representation, it has a $(k+1)$-partite representation that could be seen by simply adding a new element to every vertex  of a representing graph.

\begin{lemma}\label{partite}
Let  $H$ and $G$ be  connected graphs, each having a $k$-partite representation. 
Assume that $V(H)\cap V(G)=\{v\}$, where  $v$ that is a top vertex for both $H$ and $G$ or a bottom vertex for both $H$ and $G$. Then $H\cup G$   has a partite representation.
\end{lemma}

\begin{proof}
Let $H'$  and $G'$ be copies of $H$ and $G$  that are subgraphs of the $k$th layer in  hypercubes with a ground set $X$ and $Y$, respectively,  such that  $V(H') \cap \binom{X}{k}$ and   $V(H') \cap \binom{Y}{k}$ are the edge sets  of  $k$-partite hypergraphs  with parts $U_1,\ldots,U_k$ and parts $W_1,\ldots, W_k$, respectively.
Let $u$ and $w$  be respective copies of $v$ in $H'$ and $G'$. \\

Assume first that $v$ is a top vertex of $H$ and  of $G$. Assume that $X\cap Y=\emptyset$.
Let $F''$ be an induced subgraph of a hypercube $Q$ with ground set $X\cup Y$ and vertex set $\{x\cup w:  x\in V(H')\}\cup \{u\cup y: y\in V(G')\}$.
We see that $F''$ contains a copy $F'$ of $F$ and is contained in the  layer $2k$ of $Q$ with  the vertex $u\cup w$ playing a role of $v$.
Moreover $V(F')\cap \binom{X\cup Y}{2k}$ is an edge set of a  $2k$-partite hypergraph with parts $U_1,\ldots,U_k, W_1,\ldots, W_k$.\\

Assume now that $v$ is a bottom vertex of $H$ and of $G$. Assume further  that $X\cap Y=[k-1]$ and $u=w=[k-1]$.
Let $F''$ be an induced subgraph of a hypercube $Q$ with ground set $X\cup Y$ and vertex set $V(H') \cup V(G')$.
We see that $F''$ contains a copy $F'$ of $F$ and is contained in layer $k$ of $Q$ with vertex $v'=[k-1]$ playing a role of $v$.
Moreover $V(F')\cap \binom{X\cup Y}{k}$ is an edge set of a  $k$-partite hypergraph with parts $U_1\cup W_1,\ldots,U_k\cup W_k$.
\end{proof}

 \begin{lemma} \label{layer}
Let $Z$ be a  connected bipartite graph with a partite representation. Then for any $\gamma>0$ there is $n_1=n_1(\gamma,Z)$ such that for any $n>n_1$ the following holds. 
 Let $L_j$ be the $j$th layer of $Q_n$, where  $n/2 -n^{2/3} \leq j\leq  n/2 + n^{2/3}$. Let $G\subseteq L_j$ be a graph such that $||G||\geq \gamma ||L_j||$.  Fix one of the two  partite set of $Z$ arbitrarily and call its vertices odd. Then there is a copy of $Z$ in $G$ with odd vertices in $V_j$ and there is a copy of $Z$ in $G$ with odd vertices in $V_{j-1}$.
 \end{lemma}
 
 \begin{proof}
Let $Z$ have a $k$-partite representation in a hypercube $Q_n$  with ground set $[n]$, with  odd vertices corresponding to $k$-element sets.  We need the following notation. For any $x\in V(Q_n)$, let $\Up(x)$ be the {\it up set} of $x$, i.e., 
$\Up(x) = \{ y \subseteq [n]: ~ x\subseteq y\}$. Note that $\Up(x)$  induces a graph isomorphic to $Q_m$, for $m= n-|x|$. \\

To prove the first part of  the lemma,  for each $x\in V_{j-k}$,  let $U_x= U_{x,k}$ be  the intersection of $G$ and the $k$th edge-layer of the hypercube $Q$  induced by $\Up(x)$.  Note that $U_x$ is a subgraph of $L_j$, the $j$th layer of $Q_n$.  We call a vertex $y$ of $U_x$ a {\it full vertex} if it is in the  $k$th vertex layer of $Q$ and has degree $k$ in $U_x$, the largest possible degree. Let $u_x$ be the number of full vertices in $U_x$.
We shall argue that there is a vertex $x \in V_{j-k}$ such that $u_x$ is large.\\

Let $t$ be the number of $k$-edge stars  in $G$ with the center in $V_{j}$. Then
$$t  = \sum_{y\in V_{j}} \binom{d_G(y)}{k} \geq |V_{j}| \binom{\gamma ||L_j||/|V_{j}|}{k}.$$ 

Each such a star corresponds to a full vertex  in $U_x$, for some $x$.
Thus $\sum_{x\in V_{j-k}} u_x \geq t$ and there is a vertex $x$ such that
 $$u_x \geq  \frac{|V_{j}|}{|V_{j-k}|} \binom{\gamma' ||L_j||/|V_{j}|}{k}.$$
 Since $n/2 -n^{2/3} \leq j\leq  n/2 + n^{2/3}$ and $k$ is a fixed constant, we have that $||L_j||/||V_j|| = \frac{n}{2}(1+o(1))$ and  $|V_j|/|V_{j-k}| \geq  c'(k)$, so  
  $$u_x \geq c(k)n^k,$$ for positive constant $c(k)$ and $c'(k)$.
Consider the full vertices in $U_x$. They correspond to $u_x$ hyperedges in a $k$-uniform hypergraph with the vertex set $[n]-x$ of size  $n-j+k = \frac{n}{2} (1+o(1))$. By a theorem of Erd\H{o}s \cite{E64}, such a hypergraph contains any fixed $k$-partite $k$-uniform hypergraph, and thus in particular the one representing $Z$.  Therefore $G$ contains a copy of $Z$ with odd vertices in $V_j$.\\

To prove the second part of the Lemma, we  can repeat the above argument by considering the vertices $x$ in $V_{j+k-1}$ and their downsets. Alternatively, we see that $L_j$ corresponds to $L_{j'}$, a ``symmetric" layer, where $j'=n+1-j$, $V_{j'}$ corresponds to $V_{j-1}$, and $V_{j'-1}$ corresponds to $V_j$. Thus, finding a copy of $Z$ in $L_{j'}$ with odd vertices in $V_{j'}$ corresponds to finding a copy of $Z$ in $L_j$ with odd vertices  in $V_{j-1}$. Since $j'$ satisfies the same conditions as $j$, i.e., $n/2 -n^{2/3} \leq j' \leq n/2 + n^{2/3}$, we thus could use the first part of the Claim to obtain the second one. 
\end{proof}

 \begin{lemma} \label{final}
Let $H$ be a connected  bipartite graph with $\ell$ blocks, where  every block has  a partite representation.  Then for any  $\gamma >0$ there is $n_0=n_0(\gamma, H)$ such that for any $n>n_0$ the following holds.  Let $L_j$ be the $j$th layer of $Q_n$, where  $n/2 -n^{2/3} \leq j\leq  n/2 + n^{2/3}$. Let $G\subseteq L_j$ be a graph with $||G||= \gamma ||L_j||$.  Fix one of two  partite sets of $H$ arbitrarily and call its vertices odd. Then there is a copy of $H$ in $G$ with odd vertices in $V_j$ and there is a copy of $H$ in $G$ with odd vertices in $V_{j-1}$.
 \end{lemma}

\begin{proof}
We shall prove the statement by induction on $\ell$ with the base case $\ell=1$ directly following from Lemma \ref{layer}.\\

If $H$ has $\ell$ blocks, $\ell\geq 2$,  such that each has a $k$-partite representation, let $H=H'\cup H''$, where  $H'$ and $H''$ share a single vertex $v'$,  and $H''$ is a block of $H$, i.e., a leaf block.   Let $F$ be a graph that is a union of $q>|V(H')|$ copies  $F_1, \ldots, F_q$ of $H''$ that pairwise share only a vertex corresponding to $v'$. 
Let  $\gamma >0$ and $n_0$ be sufficiently large (we shall specify how large later). Let $G\subseteq L_j$, $||G||=\gamma ||L_j||$. \\

Assume first that $v'$ is an odd vertex.\\

The idea of the proof is as follows. We shall first find many copies of $H'$ in $G$ such that a  vertex corresponding to $v'$ is in $V_j$.  Then we shall find a copy of $F^*$ of $F$ and  a copy $H^*$ of  $H'$ such that  $V(F^*)\cap V(H^*) \cap V_j = \{v\}$, where $v$ plays a role of $v'$. Finally, since $q$ is large enough, we shall claim that there is a copy $F_i^*$ of $F_i$ for some $i$ such that 
$V(F_i^*)\cap V(H^*) \cap V_{j-1} = \emptyset$. This will imply that $F^*_i \cup H^*$ is a copy of $H$. Next we shall give  the details of this argument:\\

First, we shall construct sets $\tilde{V}, V,$ and $V'$, such that  $V'\subseteq V\subseteq \tilde{V}\subseteq V_j$ as follows: \\

Consider all copies  $H'_v$ of $H'$ in $G$ with  a vertex  playing a role of $v'$ being in $V_j$.  Let $\tilde{V}\subseteq V_j$ be the set of such vertices $v$. 
Note that $||G[V_j-\tilde{V}, V_{j-1}]|| <\frac{1}{4} \gamma||L_j||$, otherwise by induction we can find another copy $H_w'$ of $H'$ with a vertex $w\in V_j-\tilde{V}$ playing a role of $v'$, contradicting the definition of $\tilde{V}$. Here, we assume that $n_0>n_0(\gamma/4, H')$.
Note, that for each $v\in V$ there could be several copies of $H'$ with $v$ playing a role of $v'$. We choose one such copy $H'_v$ arbitrarily.
Let $V\subseteq \tilde{V}$ be a set of vertices whose degree in $G$ is at least half the average degree of $G$, i.e. for each $y\in V$,
 $$d_G(y)\geq \frac{\gamma}{2}\frac{||L_j||}{|V_j|}.$$
The total number of edges of $G$ not incident to $V$ is at most $\frac{1}{4}\gamma ||L_j||+ \frac{1}{2}||G||= \frac{3}{4}\gamma ||L_j||$. 
Thus $||G[V, V_{j-1}]|| \geq \frac{\gamma}{4}||L_j||$.  
Since  all vertices from $V_j$ in $L_j$ have the same degree, 
$$|V| \geq \frac{\gamma}{4}|V_j|.$$

Let for each $v\in V$, $A_v\subseteq V_j$ and $B_v\subseteq V_{j-1}$ be sets of vertices such that $V(H'_v)=A_v\cup B_v\cup \{v\}$, $v\not\in A_v$.
Randomly color each vertex in $V$ with red or blue independently with equal probability, color the vertices in $V_j-V$ blue. We say that a vertex  $v\in V$ is {\it good} if $v$ is red and each vertex in $A_v$ is blue. Then the expected number of good $v$'s is at least $|V|2^{-t}$, where $|A_v|=t-1$.
Thus there is a set  $V'\subseteq V$, corresponding to a set of good $v$'s, with $|V'|\geq |V|2^{-t}$, such that for each $v\in V'$,  $A_v \cap V' = \emptyset$.\\

We see that 
$$||G[ V', V_{j-1}]||\geq  \frac{\gamma}{2}\frac{||L_j||}{|V_j|}|V'| \geq \frac{\gamma}{2}\frac{||L_j||}{|V_j|} \frac{|V|}{2^{t}} \geq 
 \gamma^2 2^{-t-3} ||L_j||.$$

Consider the graph $F$ defined above. By Lemma \ref{partite} $F$ has a partite representation.   Then  by Lemma \ref{layer} there is a copy $F^*$ of $F$  in $G[V'\cup V_{j-1}]$ with a vertex $v\in V$ corresponding to $v'$.  Here, we assume that $n_0>n_1(\gamma^2 2^{-t-3}, F)$.  Let $F_i^*$ be a respective copy of $F_i$, for each $i\in [q]$. By construction of $V'$, there is a copy  $H^*$ of $H'$ in $G$ with all vertices except for $v$ not in $V'$.
Let $B_v$ be the set of vertices of $H^*$ in $V_{j-1}$.  We see that $V(F^*) \cap V(H^*) \cap V_j = \{v\}$. 
Since $q> |V(H')|>|B_v|$, there is $i \in [q]$, such that $V(F^*_i)\cap B_v\cap V_{j-1} = \emptyset$. 
Then $H^*\cup F_i^*$ is a copy of $H$ in $G$.\\

The case when $v'$ is not an odd vertex is treated similarly by first finding many copies of $H'$ with a vertex corresponding to $v'$ in $V_{j-1}$.
\end{proof}

\begin{proof}[Proof of Theorem \ref{block}]
 Let  $G'$ be a subgraph of $Q_n$ such that  $||G'|| = 2\gamma||Q_n||$, for some constant $\gamma>0$ and sufficiently large $n$.
By a standard argument, see for example Lemma 1 \cite{AMM},
$$\sum_{i: |i-n/2|>n^{2/3}} \binom{n}{i} 
= o(2^n).$$ Since the degree of each vertex in $Q_n$ is $n$,  the total number of edges in $G'$,  incident to vertices in $V_i$'s, for  $i<n/2 - n^{2/3}$ or $ i > n/2 +n^{2/3}$ is $o(n2^n) = o(||Q_n||) = o(||G||)$.
Then there is  $j\in \{n/2 - n^{2/3}, n/2 + n^{2/3}\}$ such that $L_j$  contains  at least  $\gamma ||L_i||$ edges of $G'$. 
Lemma \ref{final} applied to $G= G'\cap L_j$ concludes the proof.
\end{proof}

\begin{proof}[Proof of Theorem \ref{main}]
Let $H=H(q)$ be a union of two copies $\Theta_1$ and $\Theta_2$ of $\Theta(q)$ sharing exactly one vertex that is a main pole of $\Theta_2$ and a subdivision vertex of $\Theta_1$.\\

First we need to check that $H(q)$ is cubical.  Marquardt \cite{M} gave an explicit embedding of $H(3)$ in a hypecube. One can also use a  characterisation by Havel and Moravek \cite{HM}, who proved that a graph is cubical if and only if it has a nice edge-coloring.  Here,  an edge-coloring  is {\it nice} if any cycle uses each color an even number of times and each non-trivial path uses some color an odd number of times.  
To construct a nice coloring of $H$,  one can use nice colorings of $\Theta_1$ and $\Theta_2$ using disjoint sets of colors. These colorings were given in \cite{HM} and we show one in Figure \ref{H-graph}.\\

Next, we shall show that $H(q)$ has no partite representation.  A Hamming distance between two binary vectors $a$ and $b$, denoted $d_H(a,b)$ is the number of positions where the vectors differ.
In \cite{AMW} it is shown that  if   $q\geq 3$ and $a, a'$ are the main poles of  a copy of $\Theta(q)$ embedded in a layer of a hypercube, then  $d_H(a, a')=2$. 
Let the main poles of  $\Theta_1$ and $\Theta_2$ be denoted $v_i', v_i''$, respectively, $i=1,2$. Assume that $H$ has a $k$-partite representation for some $k$, i.e.,  $V(H)\subseteq \binom{[n]}{k}\cup \binom{[n]}{k-1}$, for some $n$ and $V(H)\cap  \binom{[n]}{k}$ corresponds to the edges of a $k$-partite hypergraph,  denote it by $\cH$. We have that $d_H(v_i', v_i'')=2$,  $i=1,2$. Since a main  pole of $\Theta_1$ and a main pole of $\Theta_2$ are adjacent,  the main poles of $\Theta_1$ are in one vertex layer and the main poles of $\Theta_2$ are in another vertex layer,  without loss of generality, $v_1', v_1'' \in \binom{[n]}{k}$ and $v_2', v_2''\in \binom{[n]}{k-1}$. 
Assume further that $v_1'$ and $v_1''$ equal to $10$ and $01$ in the first two positions of their binary representation and coincide on all other positions. Consider neighbours of $v_1$ in $\Theta_1$. At least one of these neighbours, say $w$ differs from $v_1'$ in a position that is not one of the first two  ones, say in the third position. 
Then, restricted to the first three position $v_1'$ is $101$,  $w$ is $100$, and  $v_1''$ is $011$. Thus, the remaining two vertices on the $v_1, v_1'$ path, that contains $w$, must be equal to $110$ and $010$ in these three position.  All vertices of this leg are the same in positions $4, \ldots, n$, say they are equal to $1$ on a set of positions $A\subseteq \{4, \ldots, n\}$, $|A|=k-2$.  Thus, we have that $\cH'$ contains  hyperedges $\{1, 2\}\cup A$,  $\{2,3\} \cup A$, and $\{1, 3\} \cup A$. If $\cH$ were to be $k$-partite, $1, 2, 3$ and each element of $A$ would belong to distinct parts, i.e., to $k+1$ parts, a contradiction. This shows that $H(q)$ has no partite representation. \\~\\

Finally, we shall show that $H(q)$ has a zero Tur\'an density in a hypercube. Using Theorem \ref{block} it is sufficient to show that $\Theta(q)$  has partite representation. Note that $\Theta(q) = K_{2,q}(1)$, a subdivision of $K_{2,q}$. 
We claim that  $\cH = K_{2,q}$ gives a $2$-partite representation of $\Theta(q)$. 
Indeed, let $n= q+2$ and let $\cH$ be a hypergraph on a vertex set $[n]$ with the edge set $\{ij: i\in [2], j\in [n]\setminus [2]\}$. So, $\cH$ is a bipartite graph, i.e., $2$-partite $2$-uniform hypergraph.
Let  the edges of $\cH$ correspond to subdivision vertices of $K_{2,q}(1)$. 
Let vertices $1$ and $2$ correspond to the main poles of $\Theta(q)$ and vertices $2, \ldots, q$ correspond to other vertices. 
This concludes the proof of  Theorem \ref{main}.
\end{proof}

{\bf Conclusions~~} We showed that there are cubical graphs that have no partite representation and have zero Tur\'an density in a hypercube.  
On the other hand we proved that any graph whose every block has a partite representation has a zero Tur\'an density in a hypercube.
This leads to a followup question:\\

{\bf Open question~~} Is it true that in  each cubical graph that has a zero Tur\'an density in a hypercube  each block has a partite representation?\\

{\bf Acknowledgements~~} The research was supported in part by a DFG grant FKZ AX 93/2-1.  \\

\end{document}